\documentclass[12pt]{amsart}

\usepackage{amsmath}
\usepackage{amsfonts}
\usepackage{amssymb,enumerate}
\usepackage{amsthm}
\usepackage{fancyhdr}
\usepackage{tikz}
\usetikzlibrary{arrows,matrix}
\usepackage{hyperref}

\newtheorem{lemma}{Lemma}
\newtheorem{corollary}[lemma]{Corollary}
\newtheorem{proposition}[lemma]{Proposition}
\newtheorem{theorem}[lemma]{Theorem}
\newtheorem{question}[lemma]{Question}
\theoremstyle{definition}
\newtheorem{definition}[lemma]{Definition}

\newtheorem{remark}[lemma]{Remark}

\numberwithin{lemma}{section}

\title[]{The first elements of the quotient of a numerical semigroup by a positive integer}

\author{Alessio Moscariello}



\address[Alessio Moscariello]{Dipartimento di Matematica e Informatica, \ Universit\`a di Catania, \  Viale Andrea Doria 6, 
95125 Catania,\ Italy}

\email{alessio.moscariello@studium.unict.it}

\begin{document}

\begin{abstract}
Given three pairwise coprime positive integers $a_1,a_2,a_3 \in \mathbb{Z}^+$ we show the existence of a relation between the sets of the first elements of the three quotients $\frac{\langle a_i,a_j \rangle}{a_k}$ that can be made for every $\{i.j,k\}=\{1,2,3\}$. Then we use this result to give an improved version of Johnson's semi-explicit formula for the Frobenius number $g(a_1,a_2,a_3)$ without restriction on the choice of $a_1,a_2,a_3$ and to give an explicit formula for a particular class of numerical semigroups.
\end{abstract}

\maketitle

\section*{Introduction}
Let $a_1,\ldots,a_\nu$ be a set of positive integers such that $\gcd(a_1,\ldots,a_\nu)=1$. We say that an integer $N$ is \emph{representable} if there exist $\lambda_1,\ldots,\lambda_\nu \in \mathbb{N}$ such that $$N=\sum_{i=1}^{\nu}\lambda_i a_i.$$  The well-known \emph{Diophantine Frobenius Problem} (cf. \cite{Br}) consists of finding an explicit formula for the largest non-representable integer, denoted by $g(a_1,\ldots,a_\nu)$. This problem has been widely studied: in 1884, Sylvester (cf. \cite{S}) solved the problem for $\nu=2$, giving the formula $$g(a_1,a_2)=a_1a_2-a_1-a_2.$$
However, as $\nu$ grows the problem becomes incredibly complicated: in fact today the problem is still open for $\nu \ge 3$ (\cite{RA} gathers a lot of results on this topic). The computation of a formula for the particular case $\nu=3$ and the study of concepts and invariants related to this problem (cf. \cite{RG2}) has been the main subject of a wide number of papers, and many formulas for particular triples have been found, although the main case still remains unsolved. \\
There is a strict relation (cf. \cite{J}, \cite{RA}, \cite{RG}) between this problem and the problem of finding an explicit formula for the smallest positive integer $K$ such that $Ka_3$ is representable as $Ka_3=\lambda_1a_1+\lambda_2a_2$, where $a_1,a_2,a_3$ are pairwise coprime positive integers and $\lambda_1, \lambda_2 \in \mathbb{N}$.\\
In this short paper we approach this problem by studying it in the context of \emph{quotients of numerical semigroups} (cf. \cite{RG} for a good monograph on numerical semigroups). In Section 1 we study the set that contains the first elements of a quotient of a numerical semigroup generated by two integers, by explaining the relation between the three quotients $\frac{\langle a_i,a_j \rangle}{a_k}$ that can be made with the same triple $a_1,a_2,a_3$: we find that, once the first elements of one of those three quotients is known, it is possible to deduce the first elements of the other two. Obviously the sets that we define contain the smallest non-zero element of these quotients (that we will denote with $L_i$), hence the characterization of the elements of these sets can be useful in trying to find a characterization for $L_1,L_2,L_3$. In Section 2 we give a broad generalization of the formula for $g(a_1,a_2,a_3)$ given by Johnson in \cite{J}: in particular we drop the assumptions made in \cite{J} (therefore making the formula valid for every triple) and give a more explicit formula for $g(a_1,a_2,a_3)$ by dropping some parameters from the formula, therefore obtaining an expression dependant only on $L_1,L_2,L_3$ and reducing the problem of finding a formula for $g(a_1,a_2,a_3)$ to computing those values. In Section 3 we give another application of the result showed in Section 1 to a particular class of numerical semigroups, such that of the three quotient defined is "large". Finally in Section 4 we state some remarks and ideas for a possible continuation of this work.

\section{The first elements of the quotient of a numerical semigroup}
Given two integers $m$ and $n$ with $n > 0$ we start by defining the \emph{remainder operator} $[m]_n$ as follows   $$[m]_n=\{i \in \mathbb{N} \ \ | \ \ 0 \le i < n, \ \ i \equiv m \pmod n\}$$ (i.e. the remainder of the euclidean division of $m$ by $n$). \\ Obviously if $m$ and $n$ are positive integers such that $m < n$ then $m= [m]_n$.
Furthermore, notice that this definition is strictly related to the floor function, as we have that $$m-\left \lfloor \frac{m}{n} \right \rfloor n=[m]_n.$$
The notation we use in the brackets $[\cdot]_n$ is the one normally used in $\mathbb{Z}_n$: in particular, in the bracket we denote by $[m^{-1}]_n$ the number:
$$[m^{-1}]_n=\{i \in \mathbb{N} \ \ | \ \ 0 \le i < n, \ \ im \equiv 1 \pmod n\}$$ if it exists (i.e. if $\gcd(m,n)=1$).\\

A \emph{numerical semigroup} is a submonoid $S$ of $(\mathbb{N},+)$ such that $\mathbb{N} \setminus S$ is finite. Each numerical semigroup admits a finite set of generators $a_1,\ldots,a_\nu \in \mathbb{S}$ such that $\gcd(a_1,\ldots,a_\nu)=1$ and  $$S=\{\lambda_1a_1+\ldots+\lambda_\nu a_\nu \ | \lambda_1,\ldots,\lambda_\nu \in \mathbb{N}\}.$$ In that case, we write $S=\langle a_1,\ldots,a_\nu \rangle$. \\
First, we present a result about numerical semigroups generated by two integers:
\begin{proposition}[cf. \cite{RV}, \cite{M}]\label{one}
Let $a_1,a_2$ be relatively prime positive integers. Then $$\langle a_1,a_2 \rangle = \{ x \in \mathbb{N} \ | \ [[a_2^{-1}]_{a_1}a_2x]_{(a_1a_2)} \le x \}=\{ x \in \mathbb{N} \ | \ a_2[a_2^{-1}x]_{a_1} \le x \}.$$
\end{proposition}
Given a numerical semigroup $S$ and a positive integer $d$, the \emph{quotient} $\frac{S}{d}$ is the numerical semigroup $$\frac{S}{d} := \{x \in \mathbb{N} \ | \ dx \in S \}.$$
A direct consequence of Proposition \ref{one} is the following:
\begin{corollary}\label{two}
Let $a_1,a_2 \in \mathbb{Z}^+$ be two coprime positive integers. Let $a_3 \in \mathbb{Z}^+$. Then
$$\frac{\langle a_1,a_2 \rangle}{a_3} =\{ x \in \mathbb{N} \ | \ [xa_3a_2^{-1}]_{a_1}a_2 \le xa_3 \}.$$
\end{corollary}
Let $a_1,a_2,a_3 \in \mathbb{Z}^+$ be pairwise coprime integers, such that $a_1 > a_2 > a_3 > 1$. Denote by $S$ the numerical semigroup $S=\langle a_1,a_2,a_3 \rangle$. Since $\gcd(a_i,a_j)=1$ for every $i \neq j$, then the sets $\langle a_i, a_j \rangle$ and the quotients $S_k=\frac{\langle a_i, a_j \rangle}{a_k}$ are numerical semigroups for every $\{i,j,k\}=\{1,2,3\}$.
The first definition is about the set of the first elements of these quotients:
\begin{definition}
With the notation expressed before, define for every $\{i,j\} \subset \{1,2,3\}$ with $i \neq j$ the sets $\phi_j(S_i)$ and $\tau_j(S_i)$ as the sets $\phi_j(S_i):=S_i \cap [0 , a_j]$ and  $\tau_j(S_i):=S_i \cap ]0 , a_j[$.
\end{definition}

Since $\{0, a_j\} \subset S_i$ then it is obvious that $\phi_j(S_i)=\tau_j(S_i) \cup \{0,a_j\}$. Furthermore, if $j \neq 1$ the two sets $\tau_j(S_i)$ and $\tau_j(S_1)$ are strictly related:
\begin{theorem}\label{main}
Let $a_1,a_2,a_3 \in \mathbb{Z}^+$ be pairwise coprime integers such that $a_1 > a_2 > a_3 > 1$. 
Then for every $\{j,k\}=\{2,3\}$ we have:
\begin{enumerate}
\item $\tau_j(S_k)=\{[\mu a_1a_k^{-1}]_{a_j} \ | \ \mu \in ]0,a_j[ \ \setminus  \ \tau_j(S_1)\}$;
\item $|\tau_j(S_k)|=a_j-1-|\tau_j(S_1)|$.
\end{enumerate}
\end{theorem}
\begin{proof}
 Let $x \in \tau_j(S_k)$. The fact that $j,k \neq 1$ means that $S_k=\frac{\langle a_1,a_j \rangle}{a_k}$. Now by Corollary \ref{two} we have that $$[xa_ka_1^{-1}]_{a_j}a_1 \le xa_k.$$ But $x \in \tau_j(S_k)$ means that $0< x < a_j$. By the hypothesis of pairwise coprimality of $a_1,a_2,a_3$ we obtain that $[a_1^{-1}a_k]_{a_j}$ is an invertible element modulo $a_j$. But a basic property of invertible elements in $\mathbb{Z}_{a_j}$ tells us that the function $f: \mathbb{Z} \cap [0,a_j] \rightarrow \mathbb{Z} \cap [0,a_j]$ defined by the law $f(\omega)= [\omega a_1^{-1}a_k]_{a_j}$ is a one-to-one correspondence, and that $f^{-1}(\omega)=[\omega a_k^{-1}a_1]_{a_j}$. Now we see that $$x \in \tau_j(S_k) \Leftrightarrow [xa_ka_1^{-1}]_{a_j}a_1 \le xa_k \Leftrightarrow f(x)a_1 \le xa_k=f^{-1}(f(x))a_k \Leftrightarrow$$ $$\Leftrightarrow f(x)a_1 \le [f(x)a_k^{-1}a_1]_{a_j} \Leftrightarrow f(x) \in ]0,a_j[ \ \setminus \ \tau_j(S_1).$$ 
 Then we have $$\tau_j(S_k)=\{f^{-1}(\mu) \ | \ \mu \in ]0,a_j[ \ \setminus  \ \tau_j(S_1)\}=\{[\mu a_1 a_k^{-1}]_{a_j} \ | \ \mu \in ]0,a_j[ \ \setminus  \ \tau_j(S_1)\}$$ that is the first part of our thesis. The second part is a direct consequence of the first one.
\end{proof}
This theorem will be proved useful in the next subsection.
\begin{remark}
Theorem \ref{main} tells us that $\tau_i(S_j)$ can be deduced easily from $\tau_i(S_1)$. However, it is obvious that $\tau_3(S_1) \subset \tau_2(S_1)$. Therefore, once $\tau_2(S_1)$ is known, Theorem \ref{main} allows us to obtain the two sets $\tau_2(S_3)$ and $\tau_3(S_2)$, hence obtaining the first elements of the three quotients $S_1,S_2,S_3$.
\end{remark}
\section{A formula for $g(a_1,a_2,a_3)$}
This subsection focuses on the problem of finding the \emph{Frobenius number} of a numerical semigroup generated by three integers: in particular, we improve some well-known results. Let $a_1,a_2,a_3$ be three positive integers such that $\gcd(a_1,a_2,a_3)=1$, denote by $S$ the numerical semigroup $S:=\langle a_1,a_2,a_3 \rangle$ and define the \emph{Frobenius number} $g(a_1,a_2,a_3)$ as the greatest element in $\mathbb{N} \setminus S$.
We can strenghten our hypothesis by supposing that $a_1,a_2,a_3$ are pairwise coprime. In fact Johnson (cf. \cite{J}) proved that:
\begin{theorem}[\protect{\cite[Theorem 2]{J}}]
Let $a_1,a_2,a_3 \in \mathbb{Z}^+$ and $d_{ij}=d_{ji}=\gcd(a_i,a_j)$ for every $\{i,j\} \subset \{1,2,3\}$ with $i \neq j$. Define $b_1, b_2, b_3 \in \mathbb{Z}^+$ as the three positive integers such that $a_i=b_id_{ij}d_{ik}$ for every $\{i,j,k\}=\{1,2,3\}$. Then $$g(a_1,a_2,a_3)=d_{12}d_{23}d_{31}g(b_1,b_2,b_3).$$
\end{theorem}
Let us consider three pairwise coprime positive integers $a_1,a_2,a_3 \in \mathbb{Z}^+$. \\ Define the three positive integers $L_i$ as: $$L_i := \min\{x \in \mathbb{Z}^+ \ | \ L_ia_i \in \langle a_j,a_k \rangle \} \text{  for every  } \{i,j,k\}=\{1,2,3\}.$$
Notice that $a_ja_i \in \langle a_j,a_k \rangle$ for every $\{i,j,k\}=\{1,2,3\}$, thus $L_i \le a_j$.
In the same work Johnson proved the following results:
\begin{theorem}[cf.\protect{\cite{J}}]\label{frob}
Given three pairwise coprime positive integers $a_1,a_2,a_3 \in \mathbb{Z}^+$, if $L_i > 1$ for every $i=1,2,3$ then:
\begin{enumerate}
\item The six integers $x_{ij}$, defined by the equation $$L_ia_i=x_{ij}a_j+x_{ik}a_k$$ are uniquely defined.
\item The formula $$g(a_1,a_2,a_3)=L_ia_i+\max\{x_{jk}a_k, x_{kj}a_j\}-a_1-a_2-a_3$$ holds for every $\{i,j,k\}=\{1,2,3\}$.
\end{enumerate}
\end{theorem}
The formula is strongly limited by the condition $L_i > 1$, that cuts away infinitely many triples. In our formula, we include these cases and try to remove the coefficients $x_{ij}$, using the fact that some of them are (as we will see) dependant on $L_1,L_2,L_3$. The only hypothesis (apart from pairwise coprimality) we make in the following is that $a_1>a_2>a_3>1$ (that is not restrictive). We have that:
\begin{proposition}\label{prop}
Let $a_1,a_2,a_3 \in \mathbb{Z}^+$ be three pairwise coprime integers such that $a_1>a_2>a_3>1$. Then for every $\{j,k\}=\{2,3\}$ we have that $$L_ja_j=[L_ja_ja_1^{-1}]_{a_k}a_1+[L_ja_ja_k^{-1}]_{a_1}a_k.$$
\end{proposition}
\begin{proof}
Since $L_ja_j \in \langle a_1,a_k \rangle$ then we have that there exist $\lambda_1,\lambda_k \in \mathbb{N}$ such that $$L_ja_j=\lambda_1a_1+\lambda_ka_k.$$
Moreover, since $L_j \le a_k$, then $L_ja_j \le a_ja_k < a_1a_k$, and then we must have that $\lambda_1 < a_k$ and $\lambda_k < a_1$. Considering the inequality modulo $a_1$ we obtain that $$L_ja_j \equiv \lambda_ka_k \pmod{a_1} \Rightarrow \lambda_k \equiv L_ja_ja_k^{-1} \pmod{a_1}.$$ Since $0 \le \lambda_k < a_1$ we obtain by definition of $[\cdot]_{\cdot}$ that $\lambda_k=[L_ja_ja_k^{-1}]_{a_1}$. By similar considerations modulo $a_k$ we obtain that $\lambda_1=[L_ja_ja_1^{-1}]_{a_k}$, that is our thesis. 
\end{proof}
Notice that, with the notation introduced in Section 1, we have $L_i= \min(\phi_j(S_i)\setminus\{0\})=\min(\tau_j(S_i) \cup \{a_j\})$ for every $j \neq i$. After this little remark we are ready to improve the formula of Theorem \ref{frob}:
\begin{theorem}\label{newfrob}
Let $a_1,a_2,a_3 \in \mathbb{Z}^+$ be pairwise coprime integers such that $a_1 > a_2 > a_3 > 1$. Then $$g(a_1,a_2,a_3)=L_1a_1 + \max\{ [L_2a_2a_3^{-1}]_{a_1}a_3, [L_3a_3a_2^{-1}]_{a_1}a_2 \}-a_1-a_2-a_3.$$
\end{theorem}
\begin{proof}
Denote by $R=L_1a_1 + \max\{ [L_2a_2a_3^{-1}]_{a_1}a_3, [L_3a_3a_2^{-1}]_{a_1}a_2 \}-a_1-a_2-a_3$. We prove that $R=g(a_1,a_2,a_3)$ in the two cases $L_1=1$ and $L_1 > 1$.\\
If $L_1=1$ we have that $\langle a_1,a_2,a_3 \rangle=\langle a_2,a_3 \rangle$, then $g(a_1,a_2,a_3)=g(a_2,a_3)=a_2a_3-a_2-a_3$. But since $S_1= \mathbb{N}$, then $|\tau_3(S_1)|=a_3-1$ and $|\tau_2(S_1)|=a_2-1$. By Theorem \ref{main} we obtain that $|\tau_3(S_2)|=|\tau_2(S_3)|=0$, that is $\tau_3(S_2)=\tau_2(S_3)=\emptyset$. Hence by definition of the sets $\phi_j(S_i)$ we have that $\phi_3(S_2)=\{0,a_3\}$ and $\phi_2(S_3)=\{0,a_2\}$, thus meaning $L_2=a_3$ and $L_3=a_2$. Then $$R=L_1a_1 + \max\{ [a_3a_2a_3^{-1}]_{a_1}a_3, [a_2a_3a_2^{-1}]_{a_1}a_2\}-a_1-a_2-a_3=$$ $$=a_1+\max\{[a_2]_{a_1}a_3,[a_3]_{a_1}a_2 \}-a_1-a_2-a_3=a_2a_3-a_2-a_3$$
since $[a_2]_{a_1}=a_2$ and $[a_3]_{a_1}=a_3$. Then if $L_1=1$ we have $R=g(a_1,a_2,a_3)$.\\
If $L_1 >1$ we notice that, since $a_2 < a_1$ and $(a_2,a_3)=1$, then $a_2 \not \in \langle a_1,a_3 \rangle$, i.e. $L_2 > 1$, and similarly $L_3 > 1$. Then we are under the hypotheses of Theorem \ref{frob}, therefore $$g(a_1,a_2,a_3)=L_ia_i+\max\{x_{jk}a_k, x_{kj}a_j\}-a_1-a_2-a_3$$
for every $\{i,j,k\}=\{1,2,3\}$. In particular we have 
$$g(a_1,a_2,a_3)=L_1a_1+\max\{x_{23}a_3, x_{32}a_2\}-a_1-a_2-a_3.$$
Now, since $x_{32}$ and $x_{23}$ are uniquely defined by the equations $L_2a_2=x_{21}a_1+x_{23}a_3$ and  $L_3a_3=x_{31}a_1+x_{32}a_2$ then it follows from Proposition \ref{prop} that $x_{23}=[L_2a_2a_3^{-1}]_{a_1}$ and $x_{32}=[L_3a_3a_2^{-1}]_{a_1}$ and therefore we have $R=g(a_1,a_2,a_3)$, that is our thesis.
\end{proof}
\section{Fundamental gaps}
The result obtained in Theorem \ref{main} played a marginal part in the proof of the formula for the Frobenius number in Theorem \ref{newfrob}. However, Theorem \ref{main} becomes useful once we know that $\mathbb{N}\setminus S_1$ is quite small. For this purpose we use the following known definition:
\begin{definition}
Let $S$ be a numerical semigroup. Then a gap (i.e. an element of $\mathbb{N}\setminus S$) is said to be a \emph{fundamental gap} if $\{2x,3x\} \subset S$. Denote by $FG(S)$ the set of all fundamental gaps of $S$.
\end{definition}
Notice that if $x$ is a fundamental gap of $S$, then $2x \in S$ and $3x \in S$, therefore $Kx \in S$ for every integer $K \ge 2$, i.e. $\mathbb{N} \setminus S_1 =\{1\}$. We introduce the following lemma: 
\begin{lemma}\label{Lj}
Let $a_1,a_2,a_3 \in \mathbb{Z}^+$ be three positive pairwise positive integers such that $a_1>a_2>a_3>1$. If $L_i$ is defined as in Section 2 and $L_1 = 2$ then for every $\{j,k\}=\{2,3\}$ we have
\begin{enumerate}
\item $[L_ja_ja_1^{-1}]_{a_k}=1$;
\item $[L_ja_ja_k^{-1}]_{a_1}a_k=L_ja_j-a_1$.
\end{enumerate}
   
\end{lemma}
\begin{proof}
We are going to prove that $[L_2a_2a_1^{-1}]_{a_3}=1$, since the other case follows similarly. From Proposition \ref{prop} we know that $$L_2a_2=[L_2a_2a_1^{-1}]_{a_3}a_1+[L_2a_2a_3^{-1}]_{a_1}a_3.$$ Suppose $[L_2a_2a_1^{-1}]_{a_3}=0$. By the hypothesis of pairwise coprimality this is equivalent to $L_2=a_3$, therefore meaning that $\tau_3(S_2)= \emptyset$. But by Theorem \ref{main} we must have $\tau_3(S_1)=]0,a_3[$ and consequently $L_1=1$, that is a contradiction. Hence $[L_2a_2a_1^{-1}]_{a_3}>0$.
By definition of $L_1$ we have that there exist $\lambda_2,\lambda_3 \in \mathbb{N}$ such that $2a_1=\lambda_2a_2+\lambda_3a_3$. We must have $\lambda_2 > 0$, since if $\lambda_2=0$ then $2a_1=\lambda_3a_3$, and then, for the hypothesis of pairwise coprimality, it should follow that $a_3=2$ and $a_1,a_2$ are odd integers such that $a_1>a_2$. But this is impossible, because in that case $a_1-a_2$ is an even integer, therefore $a_1-a_2=2\mu$ with $\mu \in \mathbb{Z}^+$, and then $a_1=a_2+2\mu=a_2+\mu a_3$, thus implying $L_1=1$, contradicting our hypothesis: this proves that $\lambda_2>0$. Suppose now that $[L_2a_2a_1^{-1}]_{a_3} \ge 2$. Then we have $$L_2a_2=[L_2a_2a_1^{-1}]_{a_3}a_1+[L_2a_2a_3^{-1}]_{a_1}a_3=([L_2a_2a_1^{-1}]_{a_3}-2)a_1+[L_2a_2a_3^{-1}]_{a_1}a_3+2a_1=$$ $$=([L_2a_2a_1^{-1}]_{a_3}-2)a_1+[L_2a_2a_3^{-1}]_{a_1}a_3+\lambda_2a_2+\lambda_3a_3 \Rightarrow$$ $$\Rightarrow (L_2-\lambda_2)a_2=([L_2a_2a_1^{-1}]_{a_3}-2)a_1+([L_2a_2a_3^{-1}]_{a_1}+\lambda_3)a_3 \Rightarrow (L_2-\lambda_2)a_2 \in \langle a_1,a_3 \rangle$$ and since $\lambda_2 > 0$ this contradicts the definition of $L_2$. Therefore $[L_ja_ja_1^{-1}]_{a_k}=1$, that is the first part of our statement. The second part follows easily from Theorem \ref{main}.
\end{proof}

This Lemma and Theorem \ref{main} are the main tools of the following:
\begin{corollary}
Let $a_1,a_2,a_3 \in \mathbb{Z}^+$ be three positive pairwise coprime integers such that $a_1>a_2>a_3>1$. If $a_1 \in FG(\langle a_2,a_3 \rangle)$ then
$$g(a_1,a_2,a_3)=\max\{[a_1a_2^{-1}]_{a_3}a_2,[a_1a_3^{-1}]_{a_2}a_3\}-a_2-a_3.$$
\end{corollary}
\begin{proof}
Under these hypotheses we have $L_1=2$, and from Theorem \ref{newfrob} we know that $$g(a_1,a_2,a_3)=2a_1 + \max\{ [L_2a_2a_3^{-1}]_{a_1}a_3, [L_3a_3a_2^{-1}]_{a_1}a_2 \}-a_1-a_2-a_3.$$
From the second part of Lemma \ref{Lj} it follows directly that $[L_2a_2a_3^{-1}]_{a_1}a_3=L_2a_2-a_1$ and $[L_3a_3a_2^{-1}]_{a_1}a_2=L_3a_3-a_1$, therefore our formula becomes $$g(a_1,a_2,a_3)=2a_1 + \max\{ L_2a_2-a_1, L_3a_3-a_1\}-a_1-a_2-a_3=\max\{L_2a_2,L_3a_3\}-a_2-a_3.$$ Using the notation given in Section 1, we obtain that if $a_1 \in FG(\langle a_2,a_3 \rangle)$ then $S_1=\mathbb{N} \setminus \{1\}$ and $\tau_3(S_1)=]0,a_3[ \setminus \{1\}$, leading to $]0,a_3[ \setminus \tau_3(S_1)=\{1\}$, and similarly $]0,a_2[ \setminus \tau_2(S_1)=\{1\}$.
Therefore by Theorem \ref{main} we obtain that $S_3 \cap ]0,a_2[=\tau_2(S_3)=\{[a_1a_3^{-1}]_{a_2}\}$ and $S_2 \cap ]0,a_3[=\tau_3(S_2)=\{[a_1a_2^{-1}]_{a_3}\}$, i.e. $L_3=[a_1a_3^{-1}]_{a_2}$ and $L_2=[a_1a_2^{-1}]_{a_3}$. Then we have
$$g(a_1,a_2,a_3)=\max\{[a_1a_2^{-1}]_{a_3}a_2,[a_1a_3^{-1}]_{a_2}a_3\}-a_2-a_3$$ and the proof is complete.
\end{proof}
\section{Conclusions and Questions}
The main innovation of this paper is the use of the operator $[\cdot]_{\cdot}$  in a formula for $g(a_1,a_2,a_3)$. This is meant to be an expansion of our field of action: since Curtis proved in \cite{C} the nonexistence of a closed formula that can be reduced to a set of finite polynomials, it makes sense to introduce a non-algebraic operator, such as the one we gave. Notice that our operator is strictly related to the well-known floor function by the equation $m-\left \lfloor \frac{m}{n} \right \rfloor n=[m]_n$. With this addition we are able to obtain a formula dependant only on the values $L_i$, and since the value $L_i$ can be computed by various algorithms (cf. \cite{M}, \cite{RV}) we actually obtain a procedure for computing $g(a_1,a_2,a_3)$. It is worth noting that these algorithms only use algebraic operators and the floor function, or the strictly related operator $[\cdot]_{\cdot}$. Then it makes sense to raise the following question:
\begin{question}
Is it possible to find a formula for $g(a_1,a_2,a_3)$ (or for the $L_i$) as a function of $a_1,a_2,a_3$ that is based on the algebraic operations and the floor function (or equivalently the operator $[\cdot]_{\cdot}$)?
\end{question}


\begin{thebibliography}{99}
\bibitem{Br} A. Brauer, \emph{On a problem of partitions}, Am. J. Math. 64 (1942), 299-312.

\bibitem{C} F. Curtis, \emph{On formulas for the Frobenius number of a numerical semigroup}, Math. Scand. 67 (1990), 190-192.
\bibitem{J} S. M. Johnson, \emph{A linear Diophantine Problem}, Can. J. Math., 12 (1960), 390-398.

\bibitem{M} A. Moscariello, \emph{On the smallest solution of a proportionally modular Diophantine Inequality}, Preprint.

\bibitem{RA} J. L. Ram\'irez Alfons\'in, \emph{The Diophantine Frobenius Problem}, Oxford University Press, Oxford (2005).

\bibitem{RG} J. C. Rosales, P. A. Garc\'ia Sanchez, \emph{Numerical Semigroups}, Springer (2009).

\bibitem{RG2} J. C. Rosales, P. A. Garc\'ia Sanchez, \emph{Numerical semigroups with embedding dimension three}, Archiv Math (Basel) 83 (2004), 488-496.

\bibitem{RV} J. C. Rosales, P.Vasco, \emph{The smallest integer that is solution of a proportionally modular Diophantine equation}, Mathematical Inequalities and Applications, Vol.11, Number 2 (2008), 203-212.

\bibitem{S} J. J. Sylvester, \emph{Mathematical questions with their solutions}, Educational Times, 41 (1884).
\end{thebibliography}
\end{document}